\documentclass[12pt, reqno]{amsart}

\usepackage{amsmath, amsthm, amscd, amsfonts, amssymb, graphicx, color}
\usepackage[bookmarksnumbered, colorlinks, plainpages]{hyperref}

\newtheorem{theorem}{Theorem}[section]
\newtheorem{lemma}[theorem]{Lemma}

\theoremstyle{definition}

\theoremstyle{remark}

\begin{document}
\setcounter{page}{1}

\title[Further Inequalities Associated with the Classical Gamma Function]{Further Inequalities Associated with the Classical Gamma Function}

\author[Kwara Nantomah]{Kwara Nantomah$^{*}$$^1$ }

\address{$^{1}$ Department of Mathematics, University for Development Studies, Navrongo Campus, P. O. Box 24, Navrongo, UE/R, Ghana. }
\email{\textcolor[rgb]{0.00,0.00,0.84}{mykwarasoft@yahoo.com, knantomah@uds.edu.gh}}


\subjclass[2010]{33B15, 26A48.}

\keywords{Gamma function, psi function, Inequality, generalization.}

\date{Received: xxxxxx; Revised: yyyyyy; Accepted: zzzzzz.
\newline \indent $^{*}$ Corresponding author}

\begin{abstract}
In this paper, we present some double inequalities involving certain ratios of the Gamma function. These results are further generalizations of several previous results. The approach is based on the monotonicity properties of some functions involving the the generalized Gamma functions. At the end, we pose some open problems.
\end{abstract} \maketitle

\section{Introduction and Discussion}
\noindent
Inequalities involving the classical Euler's Gamma function has gained the attension of researchers all over the world. Recent advances in this area  include those inequalities involving  ratios of the Gamma function. In \cite{Nantomah-Iddrisu-2014-aug} - \cite{K-Nantomah-14-12-aug},  the authors presented some interesting inequalities concerning such ratios, as well as some generalizations.    By utilizing similar techniques, this paper seeks to  present some new results generalizing the results of \cite{Nantomah-Iddrisu-2014-aug} - \cite{K-Nantomah-14-12-aug}. At the end, we pose some open problems involving the generalized psi functions. In the sequel, we recall some basic definitions concerning the Gamma function and its generalizations. These definitions are required in order to establish our results.\\


\noindent
The well-known classical Gamma function, $\Gamma(t)$ and the classical psi or digamma function, $\psi(t)$ are usually defined for $t>0$ by:
\begin{equation*}\label{eqn:gamma-and-digamma}
\Gamma(t)=\int_0^\infty e^{-x}x^{t-1}\,dx  \quad \text{and} \quad \psi(t)=\frac{d}{dt}\ln \Gamma(t) =\frac{\Gamma'(t)}{\Gamma(t)}.
\end{equation*}

\noindent
The $p$-Gamma function, $\Gamma_p(t)$  and the $p$-psi function, $\psi_p(t)$ are also defined for $p\in N$ and $t>0$ by:
\begin{equation*}\label{eqn:p-gamma-and-p-digamma}
\Gamma_p(t)=\frac{p!p^t}{t(t+1) \dots (t+p)} \quad \text{and} \quad \psi_p(t)=\frac{d}{dt}\ln \Gamma_p(t) =\frac{\Gamma_p'(t)}{\Gamma_p(t)}
\end{equation*}
where \,  $\Gamma_p(t) \rightarrow \Gamma(t)$ \,  and  \,   $\psi_p(t) \rightarrow \psi(t)$ as $p \rightarrow \infty$. For more information on this function, see \cite{Krasniqi-Mansour-Shabani-2010} and the references therein. \\

\noindent
Also, the $q$-Gamma function, $\Gamma_q(t)$  and the $q$-psi function, $\psi_q(t)$ are defined for $q\in(0,1)$ and $t>0$ by:
\begin{equation*}\label{eqn:q-gamma-and-q-digamma}
\Gamma_q(t)=(1-q)^{1-t}\prod_{n=1}^{\infty}\frac{1-q^n}{1-q^{t+n}} \quad \text{and} \quad \psi_q(t)=\frac{d}{dt}\ln \Gamma_q(t) =\frac{\Gamma_q'(t)}{\Gamma_q(t)}
\end{equation*}
where \,  $\Gamma_q(t) \rightarrow \Gamma(t)$ \,  and  \,   $\psi_q(t) \rightarrow \psi(t)$ as $q \rightarrow 1^-$. See also \cite{Mansour-2008}.\\

\noindent
D\'{i}az and  Pariguan \cite{Diaz-Pariguan-2007} in 2007 further defined the $k$-Gamma function, $\Gamma_k(t)$  and the $k$-psi function, $\psi_k(t)$ for $k>0$ and $t>0$ as follows: 
\begin{equation*}\label{eqn:k-gamma-and-k-digamma}
\Gamma_k(t)=\int_0^\infty e^{-\frac{x^k}{k}}x^{t-1}\,dx \quad \text{and} \quad \psi_k(t)=\frac{d}{dt}\ln \Gamma_k(t) =\frac{\Gamma_k'(t)}{\Gamma_k(t)}
\end{equation*}
where \,  $\Gamma_k(t) \rightarrow \Gamma(t)$ \,  and  \,   $\psi_k(t) \rightarrow \psi(t)$ as $k \rightarrow 1$.\\

\noindent
Similarly in 2005, D\'{i}az and Teruel \cite{Diaz-Teruel-2005} defined the  $(q,k)$-Gamma and the $(q,k)$-psi functions for $q\in(0,1)$, $k>0$ and $t>0$ by:
\begin{equation*}\label{eqn:(q,k)-gamma-and-(q,k)-digamma}
\Gamma_{(q,k)}(t)=\frac{  (1-q^{k})_{q,k}^{\frac{t}{k}-1}}  {(1-q)^{\frac{t}{k}-1}}
\quad \text{and} \quad \psi_{(q,k)}(t)=\frac{d}{dt}\ln \Gamma_{(q,k)}(t) =\frac{\Gamma'_{(q,k)}(t)}{\Gamma_{(q,k)}(t)}
\end{equation*}
where \,  $(t)_{n,k}=t(t+k)(t+2k) \dots (t+(n-1)k)=\prod _{j=0}^{n-1}(t+jk)$ \, is the $k$-generalized Pochhammer symbol and  \, $\Gamma_{(q,k)}(t) \rightarrow \Gamma(t)$, \,  $\psi_{(q,k)}(t) \rightarrow \psi(t)$ as $q \rightarrow 1^-$,  $k \rightarrow 1$.\\

\noindent
Furthermore in 2012, Krasniqi and Merovci \cite{Krasniqu-Merovci-2012} defined the $(p,q)$-Gamma and the $(p,q)$-psi functions for $p\in N$, $q\in(0,1)$ and $t>0$ by: 
\begin{equation*}\label{eqn:(p,q)-gamma-and-(p,q)-digamma}
\Gamma_{(p,q)}(t)=\frac{[p]_{q}^{t}[p]_{q}!}{[t]_{q}[t+1]_{q}\dots[t+p]_{q}} \quad \text{and} \quad \psi_{(p,q)}(t)=\frac{d}{dt}\ln \Gamma_{(p,q)}(t) =\frac{\Gamma'_{(p,q)}(t)}{\Gamma_{(p,q)}(t)}
\end{equation*}
where \, $[p]_{q}=\frac{1-q^p}{1-q}$, \,and\,  $\Gamma_{(p,q)}(t) \rightarrow \Gamma(t)$, \, $\psi_{(p,q)}(t) \rightarrow \psi(t)$ as $p \rightarrow \infty$,  $q \rightarrow 1^-$.\\

\noindent
As defined above, the generalized psi functions $\psi_p(t)$, $\psi_q(t)$, $\psi_k(t)$, $\psi_{(p,q)}(t)$ and $\psi_{(q,k)}(t)$ possess the following series characterizations (see also \cite{K-Nantomah-14-malaya}, \cite{K-Nantomah-14-12-aug} and the references therein).
\begin{align}
\psi_p(t)&= \ln p - \sum_{n=0}^{p}\frac{1}{n+t}  \label{eqn:series-p-psi-further}\\
\psi_q(t)&= -\ln(1-q) + \ln q \sum_{n=1}^{\infty}\frac{q^{nt}}{1-q^{n}}  \label{eqn:series-q-psi-further}\\
\psi_k(t)&=\frac{\ln k-\gamma}{k}-\frac{1}{t}+\sum_{n=1}^{\infty} \frac{t}{nk(nk+t)}  \label{eqn:series-k-psi-further}\\
\psi_{(p,q)}(t)&= \ln[p]_q  + (\ln q)\sum_{n=1}^{p}\frac{q^{nt}}{1-q^{n}}  \label{eqn:series-(p,q)-psi-further}\\
\psi_{(q,k)}(t)&=\frac{-\ln(1-q)}{k} + (\ln q)\sum_{n=1}^{\infty}\frac{q^{nkt}}{1-q^{nk}} \label{eqn:series-(q,k)-psi-further}
\end{align}
with $\gamma=\lim_{n \rightarrow \infty} \left( \sum_{k=1}^{n}\frac{1}{k}- \ln n \right)= 0.577215664...$ denoting the Euler-Mascheroni's constant.\\

\section{Results}
\noindent
We now present our results. Let us begin with the following Lemmas pertaining to our results.
\begin{lemma}\label{lem-gen:q-psi_minus_(p,q)-psi-further}
Assume that $\lambda \geq \mu>0$,  $p\in N$, $q\in(0,1)$ and $g(t)>0$. Then,
\begin{equation*}\label{eqn-gen:q-psi_minus_(p,q)-psi-further} 
\lambda \ln(1-q) +\mu \ln[p]_{q}+ \lambda \psi_q(g(t)) - \mu \psi_{(p,q)}(g(t)) \leq0.
\end{equation*}
\end{lemma}

\begin{proof}
From  ~(\ref{eqn:series-q-psi-further}) and ~(\ref{eqn:series-(p,q)-psi-further}) we have,
\begin{multline*} 
\lambda \ln(1-q) +\mu \ln[p]_{q} + \lambda \psi_q(g(t)) - \mu \psi_{(p,q)}(g(t)) \\= (\ln q) \left[ \lambda \sum_{n=1}^{\infty}\frac{q^{ng(t)}}{1-q^{n}} - \mu \sum_{n=1}^{p}\frac{q^{ng(t)}}{1-q^{n}} \right] \leq0.
\end{multline*}
\end{proof}

\begin{lemma}\label{lem-gen:q-psi_minus_(q,k)-psi-further}
Assume that $\lambda \geq \mu>0$,  $q\in(0,1)$, $k\geq1$ and $g(t)>0$. Then,
\begin{equation*}\label{eqn-gen:q-psi_minus_(q,k)-psi-further} 
\lambda \ln(1-q) - \mu \frac{\ln(1-q) }{k}+ \lambda \psi_q(g(t)) - \mu \psi_{(q,k)}(g(t)) \leq 0 .
\end{equation*}
\end{lemma}

\begin{proof}
From  ~(\ref{eqn:series-q-psi-further}) and ~(\ref{eqn:series-(q,k)-psi-further}) we have,
\begin{multline*}
\lambda \ln(1-q) - \mu \frac{\ln(1-q) }{k}+ \lambda \psi_q(g(t)) - \mu \psi_{(q,k)}(g(t)) \\ = (\ln q) \sum_{n=1}^{\infty} \left[ \lambda \frac{q^{ng(t)}}{1-q^{n}} - \mu \frac{q^{nkg(t)}}{1-q^{nk}} \right] \leq0.
\end{multline*}
\end{proof}

\begin{lemma}\label{lem-gen:k-psi_minus_(p,q)-psi-further}
Assume that  $\lambda>0$,  $\mu>0$,  $k>0$, $p\in N$,  $q\in(0,1)$ and $g(t)>0$. Then,
\begin{equation*}\label{eqn-gen:k-psi_minus_(p,q)-psi-further} 
\mu \ln[p]_{q} - \frac{\lambda \ln k}{k} + \frac{\lambda \gamma}{k}   + \frac{\lambda }{g(t)}+\lambda \psi_k(g(t)) - \mu \psi_{(p,q)}(g(t)) >0.
\end{equation*}
\end{lemma}

\begin{proof}
From  ~(\ref{eqn:series-k-psi-further}) and ~(\ref{eqn:series-(p,q)-psi-further}) we have,
\begin{multline*}
\mu \ln[p]_{q} - \frac{\lambda \ln k}{k} + \frac{\lambda \gamma}{k}   + \frac{\lambda }{g(t)}+\lambda \psi_k(g(t)) - \mu \psi_{(p,q)}(g(t)) \\ = \lambda \sum_{n=1}^{\infty}\frac{g(t)}{nk(nk+g(t))} -\mu (\ln q) \sum_{n=1}^{p}\frac{q^{ng(t)}}{1-q^{n}}>0.
\end{multline*}
\end{proof}

\begin{lemma}\label{lem-gen:k-psi_minus_(q,k)-psi-further}
Assume that    $\lambda>0$,  $\mu>0$,  $q\in(0,1)$,  $k>0$  and $g(t)>0$. Then,
\begin{equation*}\label{eqn-gen:k-psi_minus_(q,k)-psi-further} 
\frac{\lambda \gamma}{k} + \frac{\lambda}{g(t)} - \frac{\ln(k^{\lambda}(1-q)^{\mu})}{k} +  \lambda \psi_k(g(t)) - \mu \psi_{(q,k)}(g(t))>0.
\end{equation*}
\end{lemma}

\begin{proof}
From  ~(\ref{eqn:series-k-psi-further}) and ~(\ref{eqn:series-(q,k)-psi-further}) we have,
\begin{multline*}
\frac{\lambda \gamma}{k} + \frac{\lambda}{g(t)} - \frac{\ln(k^{\lambda}(1-q)^{\mu})}{k} +  \lambda \psi_k(g(t)) - \mu \psi_{(q,k)}(g(t)) \\ = \lambda \sum_{n=1}^{\infty}\frac{g(t)}{nk(nk+g(t))} -  \mu (\ln q) \sum_{n=1}^{\infty}\frac{q^{nkg(t)}}{1-q^{nk}}>0.
\end{multline*}
\end{proof}

\begin{theorem}\label{thm-gen:funct-q-gamma_(p,q)-gamma}
Let $g(t)$ be a positive, increasing and differentiable function,  $p\in N$ and $q\in(0,1)$. Then for positive real numbers $\lambda$ and $\mu$ such that $\lambda \geq \mu$, the inequalities: 
\begin{equation}\label{eqn-gen:ineq-ratio-1}
\frac{(1-q)^{\lambda(g(0)-g(x))}\Gamma_q(g(0))^\lambda}{ [p]^{-\mu(g(0)-g(x))}_{q}\Gamma_{(p,q)}(g(0))^\mu} \geq 
\frac{\Gamma_q(g(x))^\lambda}{\Gamma_{(p,q)}(g(x))^\mu} \geq
\frac{(1-q)^{\lambda(g(y)-g(x))}\Gamma_q(g(y))^\lambda}{ [p]^{-\mu(g(y)-g(x))}_{q}\Gamma_{(p,q)}(g(y))^\mu}
\end{equation}
are valid for $0<x<y$. 
\end{theorem}
\begin{proof}
Define a function $G$ for $p \in N$ and $q\in(0,1)$ by 
\begin{equation*}\label{eqn-gen:funct-q-gamma_(p,q)-gamma} 
G(t)=\frac{(1-q)^{\lambda g(t)}\Gamma_q(g(t))^\lambda}{ [p]^{-\mu g(t)}_{q}\Gamma_{(p,q)}(g(t))^\mu}, \quad t\in (0,\infty).
\end{equation*}
Let $u(t)=\ln G(t)$. Then,
\begin{align*}
u(t) &=\ln \frac{(1-q)^{\lambda g(t)}\Gamma_q(g(t))^\lambda}{ [p]^{-\mu g(t)}_{q}\Gamma_{(p,q)}(g(t))^\mu} \\
&= \lambda g(t)\ln(1-q) + \mu g(t) \ln [p]_{q}+ \lambda \ln \Gamma_q(g(t)) - \mu \ln \Gamma_{(p,q)}(g(t)).\\
\text{Then,} &  \\
u'(t)&=\lambda g'(t) \ln(1-q)+ \mu g'(t) \ln[p]_{q} + \lambda g'(t) \psi_q(g(t)) - \mu g'(t) \psi_{(p,q)}(g(t))\\
      &=g'(t) \left[ \lambda \ln(1-q)+ \mu \ln[p]_{q} + \lambda \psi_q(g(t)) - \mu \psi_{(p,q)}(g(t)) \right] \leq 0
\end{align*}
\noindent
as a consequence of Lemma~\ref{lem-gen:q-psi_minus_(p,q)-psi-further}. That implies $u$ is non-increasing on $t\in(0,\infty)$. Hence $G=e^{u(t)}$ is non-increasing  and  for  $0<x<y$ we have, 
\begin{equation*}
G(0) \geq G(x) \geq G(y) 
\end{equation*}
establishing the result.
\end{proof}

\begin{theorem}\label{thm-gen:funct-q-gamma_(q,k)-gamma}
Let $g(t)$ be a positive, increasing and differentiable function,   $q\in(0,1)$ and $k\geq1$. Then for positive real numbers $\lambda$ and $\mu$ such that $\lambda \geq \mu$, the inequalities: 
\begin{equation}\label{eqn-gen:ineq-ratio-2}
\frac{(1-q)^{\lambda(g(0)-g(x))} \Gamma_q(g(0))^\lambda}{ (1-q)^{\frac{\mu}{k}(g(0)-g(x))}\Gamma_{(q,k)}(g(0))^\mu} \geq 
\frac{\Gamma_q(g(x))^\lambda}{\Gamma_{(q,k)}(g(x))^\mu} \geq
\frac{(1-q)^{\lambda(g(y)-g(x))} \Gamma_q(g(y))^\lambda}{ (1-q)^{\frac{\mu}{k}(g(y)-g(x))}\Gamma_{(q,k)}(g(y))^\mu}  
\end{equation}
are valid for $0<x<y$. 
\end{theorem}

\begin{proof}
Define a function $H$ for $q\in(0,1)$ and $k\geq1$ by 
\begin{equation*}\label{eqn-gen:funct-q-gamma_(q,k)-gamma} 
H(t)=\frac{(1-q)^{\lambda g(t)}\Gamma_q(g(t))^\lambda}{ (1-q)^{\frac{\mu g(t)}{k}}\Gamma_{(q,k)}(g(t))^\mu}, \quad t\in (0,\infty).
\end{equation*}
Let $v(t)=\ln H(t)$. Then,
\begin{align*}
v(t) &=\ln \frac{(1-q)^{\lambda g(t)}\Gamma_q(g(t))^\lambda}{ (1-q)^{\frac{\mu g(t)}{k}}\Gamma_{(q,k)}(g(t))^\mu}\\
&=\lambda g(t) \ln(1-q) - \frac{\mu g(t)}{k}\ln(1-q)+ \lambda \ln \Gamma_q(g(t)) - \mu \ln \Gamma_{(q,k)}(g(t)).\\
\text{Then,} &  \\
v'(t)&=\lambda g'(t) \ln(1-q) - \frac{\mu g'(t)}{k}\ln(1-q) + \lambda g'(t) \psi_q(g(t)) - \mu g'(t) \psi_{(q,k)}(g(t)) \\
       &=g'(t) \left[ \lambda \ln(1-q) - \mu \frac{\ln(1-q) }{k}+ \lambda \psi_q(g(t)) - \mu \psi_{(q,k)}(g(t)) \right] \leq 0
\end{align*}
\noindent
as a result of Lemma~\ref{lem-gen:q-psi_minus_(q,k)-psi-further}.  That implies $v$ is non-increasing on $t\in(0,\infty)$. Hence $H=e^{v(t)}$ is non-increasing  and  for  $0<x<y$ we have, 
\begin{equation*}
H(0) \geq H(x) \geq H(y) 
\end{equation*}
yielding the result.
\end{proof}

\begin{theorem}\label{thm-gen:funct-k-gamma_(p,q)-gamma}
Let $g(t)$ be a positive, increasing and differentiable function,   $k>0$, $p\in N$ and  $q\in(0,1)$. Then for positive real numbers $\lambda$ and $\mu$, the inequalities: 
\begin{multline}\label{eqn-gen:ineq-ratio-3} 
\frac{(g(0))^{\lambda} k^{-\frac{\lambda}{k}(g(0)-g(x))}e^{\frac{\lambda \gamma}{k}(g(0)-g(x))}\Gamma_k(g(0))^\lambda}{(g(x))^{\lambda} [p]^{-\mu(g(0)-g(x))}_{q}\Gamma_{(p,q)}(g(0))^\mu} < 
\frac{\Gamma_k(g(x))^\lambda}{\Gamma_{(p,q)}(g(x))^\mu} \\ <
\frac{(g(y))^{\lambda} k^{-\frac{\lambda}{k}(g(y)-g(x))}e^{\frac{\lambda \gamma}{k}(g(y)-g(x))}\Gamma_k(g(y))^\lambda}{(g(x))^{\lambda} [p]^{-\mu(g(y)-g(x))}_{q}\Gamma_{(p,q)}(g(y))^\mu} 
\end{multline}
are valid for $0<x<y$. 
\end{theorem}

\begin{proof}
Define a function $S$ for $k>0$, $p \in N$ and $q\in(0,1)$  by 
\begin{equation*}\label{eqn-gen:funct-k-gamma_(p,q)-gamma} 
S(t)=\frac{(g(t))^{\lambda}k^{-\frac{\lambda g(t)}{k}}e^{\frac{\lambda \gamma g(t)}{k}}\Gamma_k(g(t))^\lambda}{ [p]^{-\mu g(t)}_{q}\Gamma_{(p,q)}(g(t))^\mu}, \quad t\in (0,\infty).
\end{equation*}
Let $w(t)=\ln S(t)$. Then,
\begin{align*}
w(t) &=\ln \frac{(g(t))^{\lambda}k^{-\frac{\lambda g(t)}{k}}e^{\frac{\lambda \gamma g(t)}{k}}\Gamma_k(g(t))^\lambda}{ [p]^{-\mu g(t)}_{q}\Gamma_{(p,q)}(g(t))^\mu}\\
&= \mu g(t)\ln [p]_{q} - \frac{\lambda g(t)}{k}\ln k  + \frac{\lambda \gamma g(t)}{k} + \lambda \ln(g(t))  \\
& \quad + \lambda \ln \Gamma_k(g(t)) - \mu \ln \Gamma_{(p,q)}(g(t)).\\
\text{Then,}& \\
w'(t)&= \mu g'(t) \ln[p]_{q} -  \frac{\lambda g'(t) \ln k}{k} + \frac{\lambda \gamma g'(t)}{k}   + \lambda \frac{g'(t)}{g(t)}\\
& \quad + \lambda g'(t) \psi_k(g(t)) - \mu g'(t)\psi_{(p,q)}(g(t))\\
&=g'(t)  \left[ \mu \ln[p]_{q} - \frac{\lambda \ln k}{k} + \frac{\lambda \gamma}{k}   + \frac{\lambda}{g(t)}+ \lambda \psi_k(g(t)) - \mu \psi_{(p,q)}(g(t))  \right]>0 
\end{align*}
\noindent
as a result of  Lemma~\ref{lem-gen:k-psi_minus_(p,q)-psi-further}. That implies $w$ is increasing on $t\in(0,\infty)$. Hence $S=e^{w(t)}$ is increasing  and  for  $0<x<y$ we have, 
\begin{equation*}
S(0) < S(x) < S(y) 
\end{equation*}
completing the proof.
\end{proof}

\begin{theorem}\label{thm-gen:funct-k-gamma_(q,k)-gamma}
Let $g(t)$ be a positive, increasing and differentiable function,   $k>0$ and  $q\in(0,1)$. Then for positive real numbers $\lambda$ and $\mu$, the inequalities: 

\begin{multline}\label{eqn-gen:ineq-ratio-4}
\frac{(g(0))^{\lambda}  e^{\frac{\lambda \gamma }{k}(g(0)-g(x))} \Gamma_k(g(0))^\lambda}{ (g(x))^{\lambda} k^{\frac{\lambda}{k}(g(0)-g(x))}(1-q)^{\frac{\mu}{k}(g(0)-g(x))}\Gamma_{(q,k)}(g(0))^\mu} < 
\frac{\Gamma_k(g(x))^\lambda}{\Gamma_{(q,k)}(g(x))^\mu}\\  <
\frac{(g(y))^{\lambda}  e^{\frac{\lambda \gamma }{k}(g(y)-g(x))} \Gamma_k(g(y))^\lambda}{ (g(x))^{\lambda} k^{\frac{\lambda}{k}(g(y)-g(x))}(1-q)^{\frac{\mu}{k}(g(y)-g(x))}\Gamma_{(q,k)}(g(y))^\mu} 
\end{multline}
are valid for $0<x<y$. 
\end{theorem}

\begin{proof}
Define a function $T$ for  $k>0$ and $q\in(0,1)$ by 
\begin{equation*}\label{eqn-gen:funct-k-gamma_(q,k)-gamma} 
T(t)=\frac{(g(t))^{\lambda} e^{\frac{\lambda \gamma g(t)}{k}}\Gamma_k(g(t))^\lambda}{ k^{\frac{\lambda g(t)}{k}}(1-q)^{\frac{\mu g(t)}{k}}\Gamma_{(q,k)}(g(t))^\mu}, \quad t\in (0,\infty).
\end{equation*}
Let $\delta(t)=\ln T(t)$. Then,
\begin{align*}
\delta(t) &=\ln \frac{(g(t))^{\lambda} e^{\frac{\lambda \gamma g(t)}{k}}\Gamma_k(g(t))^\lambda}{ k^{\frac{\lambda g(t)}{k}}(1-q)^{\frac{\mu g(t)}{k}}\Gamma_{(q,k)}(g(t))^\mu} \\
&=\lambda \ln(g(t)) + \frac{\lambda \gamma g(t)}{k} - \frac{\lambda g(t)}{k}\ln k - \frac{\mu g(t)}{k}\ln(1-q) \\
&\quad + \lambda \ln \Gamma_k(g(t))  - \mu \ln \Gamma_{(q,k)}(g(t)). \\
 \text{Then,}& \\
\delta'(t)&=\frac{\lambda \gamma g'(t)}{k} + \lambda \frac{g'(t)}{g(t)}- \frac{g'(t) \ln(k^{\lambda}(1-q)^{\mu})}{k} +  \lambda g'(t) \psi_k(g(t)) - \mu g'(t) \psi_{(q,k)}(g(t))\\
&=g'(t)  \left[\frac{\lambda \gamma}{k} + \frac{\lambda}{g(t)} -\frac{\ln(k^{\lambda}(1-q)^{\mu})}{k} +  \lambda \psi_k(g(t)) - \mu \psi_{(q,k)}(g(t)) \right]>0
\end{align*}
\noindent
as a result of Lemma~\ref{lem-gen:k-psi_minus_(q,k)-psi-further}. That implies $\delta$ is increasing on $t\in(0,\infty)$.  Hence $T=e^{\delta(t)}$ is increasing  and  for  $0<x<y$ we have,  
\begin{equation*}
T(0) < T(x) < T(y) 
\end{equation*}
and the result follows.
\end{proof}

\section{Concluding Remarks}
\noindent
In particular, if we let $g(t)=\alpha+\beta t$ on the interval $0<t<1$, then we recover the entire results of \cite{K-Nantomah-14-12-aug}. Also, by setting $g(t)=\alpha + t$ and $\lambda = \mu = 1$ on $0<t<1$, we obtain the results of \cite{K-Nantomah-14-malaya}. The results \cite{Nantomah-Iddrisu-2014-aug} - \cite{K-Nantomah-14-12-aug} are therefore special cases of the results of this paper.

\section{Open Problems}
\noindent
For $k>0$, $p\in N$ and $q\in(0,1)$, let  $\psi_p(t)$, $\psi_q(t)$, $\psi_{(p,q)}(t)$ and $\psi_{(q,k)}(t)$ respectively be the $p$, $q$, $(p,q)$ and $(q,k)$ analogues of the classical psi function, $\psi(t)$. \\
\noindent
\textbf{Problem 1:} Under what conditions will the  statements:
\begin{equation*}
\ln p + \ln(1-q) + \psi_q(t) - \psi_{p}(t) = \sum_{n=0}^{p}\frac{1}{n+t} +  (\ln q) \sum_{n=1}^{\infty}\frac{q^{nt}}{1-q^{n}}\leq(\geq)0
\end{equation*}
be valid?\\
\noindent
\textbf{Problem 2:} Under what conditions will the  statements:
\begin{equation*}
-\ln [p]_q -\frac{\ln(1-q)}{k} + \psi_{(p,q)}(t) - \psi_{(q,k)}(t) = (\ln q) \left[\sum_{n=1}^{p}\frac{q^{nt}}{1-q^{n}} -   \sum_{n=1}^{\infty}\frac{q^{nkt}}{1-q^{nk}} \right] \leq(\geq)0
\end{equation*}
be valid?

\bibliographystyle{plain}


\end{document}